\documentclass[12pt, reqno]{amsart}

\usepackage{amsmath, amsthm, amscd, amsfonts, amssymb, graphicx, color}
\usepackage{bbold}
\usepackage[bookmarksnumbered, colorlinks, plainpages]{hyperref}
\hypersetup{colorlinks=true,linkcolor=red, anchorcolor=green, citecolor=cyan, urlcolor=red, filecolor=magenta, pdftoolbar=true}

\newtheorem{theorem}{Theorem}[section]
\newtheorem{lemma}[theorem]{Lemma}
\newtheorem{proposition}[theorem]{Proposition}

\theoremstyle{definition}
\newtheorem{definition}[theorem]{Definition}
\newtheorem{example}[theorem]{Example}

\theoremstyle{remark}
\newtheorem{remark}[theorem]{Remark}
\numberwithin{equation}{section}

\begin{document}

\setcounter{page}{1}

\title[Nonlinear elliptic equations with variable exponents]{Nonlinear elliptic equations with variable exponents satisfying Cerami condition}

\author[O. Benslimane, A. Aberqi and J. Bennouna]{Omar Benslimane$^1$$^{*}$ Ahmed Aberqi$^2$ \MakeLowercase{and} Jaouad Bennouna$^1$}

\address{$^{1}$Sidi Mohamed Ben Abdellah University, Faculty of Sciences Dhar Al Mahraz, Department of Mathematics, B.P 1796 Atlas Fez, Morocco.}
\email{\textcolor[rgb]{0.00,0.00,0.84}{omar.benslimane@usmba.ac.ma}}
\email{\textcolor[rgb]{0.00,0.00,0.84}{jbennouna@hotmail.fr}}

\address{$^{2}$Sidi Mohamed Ben Abdellah University, National School of Applied Sciences Fez, Morocco.}
\email{\textcolor[rgb]{0.00,0.00,0.84}{aberqi\_ahmed@yahoo.fr}}

%\dedicatory{This paper is dedicated to Professor ABCD}

%\let\thefootnote\relax\footnote{Copyright 2016 by the Tusi Mathematical Research Group.}



\date{Received: xxxxxx; Accepted: zzzzzz.
\newline \indent $^{*}$Corresponding author}

\begin{abstract}
We are concerned with the study of the existence and multiplicity of solutions for Dirichlet boundary value problems involving the $( p( x ), \, q( x ) )-$ equation and the nonlinearity is superlinear but does not satisfy the usual Ambrossetti-Rabinowitz condition in the framework of Sobolev spaces with variable exponents in Complete manifolds. The main results are established by means of the mountain pass theorem and Fountain theorem with Cerami condition. Moreover, we are giving an example of a $( p( x ), \, q( x ) )$ equation that verifies all our demonstrated results.
\end{abstract} \maketitle

\section{Introduction}

Let $ ( M, g ) $ be a complete non-compact Riemannian manifold, we consider the following problem
$$ ( \mathcal{P} ) \, \begin{cases}
- \, \Delta_{p( x )} u( x ) - \Delta_{q( x )} u( x ) = \, f( x, u( x ) )  & \text{in \,\,M }, \\[0.3cm]
\, u \,  = \, 0  & \text{on \,$\partial$M },
\end{cases} $$ 
where the variable exponents $p, \, q \in \overline{M} \rightarrow ( 1, \, + \infty)$ are continuous functions, satisfy the following assumption:
\begin{equation}\label{0}
1 < q^{-} \leq q^{+} < p^{-} \leq p^{+} < \infty,
\end{equation}
with $$ q^{+} = \sup_{x \in \overline{M}} q( x ), \,\, q^{-} = \inf_{x \in \overline{M}} q( x ), \,\, p^{+} = \sup_{x \in \overline{M}} p( x ), \,\, \mbox{and} \,\, p^{-} = \inf_{x \in \overline{M}} p( x ).$$ The operator $\Delta_{p( x )} u( x ) = - \mbox{div} ( |\, \nabla u( x )\,|^{p( x ) - 2} \nabla u( x ) ) $ and $\Delta_{q( x )} u( x ) = \\- \mbox{div} ( |\, \nabla u( x )\,|^{q( x ) - 2} \nabla u( x ) ) $ are called the $p( x )-$Laplacian and $q( x )-$Laplacian in $( M, \, g )$.\\
Before giving our hypotheses on the nonlinearity $f.$ We first define the family of functions
$$\mathcal{F} = \{ \, G_{\lambda} ( x, t ) = f( x, t ) t - \lambda F( x, t ), \,\, \lambda \in [ 2q^{-}, \, 2p^{+} ] \, \},$$
where $F( x, t ) = \displaystyle\int_{0}^{t} f( x, s )\,\, ds$. Note that when $p( x ) = p$ is constant, the family $\mathcal{F}$ consists of only one element, that is, 
$$\mathcal{F} = \{\, G_{p} ( x, t ) = f( x, t )t - p\, F( x, t ) \,\}.$$
Throughout this paper, we assume the following hypotheses on the Carathéodory function $ f: M \times \mathbb{R} \rightarrow \mathbb{R}$
\begin{itemize}
\item[($f_{1})$:] $|\, f( x, t )\,| \leq c_{1} + c_{2} \,|\, t\,|^{r( x ) -1}$ for all $( x, t ) \in M \times \mathbb{R}$ 
where $r: \overline{M} \rightarrow (1, \, +\infty )$ is a bounded continuous function such that $ r( x ) < q^{*} ( x ) = \frac{N\, q( x )}{N - q( x )}$ for any $ x \in \overline{M}$.
\item[($f_{2})$:] $\displaystyle\lim_{|\,t\,| \rightarrow \infty} \frac{f( x, t ) \, t}{|\, t\,|^{p^{+}}} = +\infty$ uniformly for a.e $ x \in M$.
\item[($f_{3})$:] $ f( x, t ) = 0 ( |\, t\,|^{p^{+} - 1} )$ as $ t \rightarrow 0$ uniformly for $ x \in M $.
\item[($f_{4})$:] There exists a constant $ \theta \geq 1 $ such that for any $ \beta \in [ 0, \, 1], \, t \in \mathbb{R}$ and for all $ G_{\lambda} \in \mathcal{F}, \, \eta \in [ 2q^{-}, \, 2p^{+} ]$ the inequality  $$ \theta \, G_{\lambda} ( x, t ) \geq G_{\eta} ( x, \beta t ) \,\, \mbox{holds for a.e} \,\, x \in M$$
\item[($f_{5})$:] $ f( x, -t ) = - f( x, t ) $ for all $( x, t ) \in M \times \mathbb{R}.$  
\end{itemize}

There are many papers done by authors in this context, and we start by a pioneer work for multiplicity results with superlinear nonlinearities was published in the classical Sobolev space when $ p( x ) = p = \mbox{cste}$ by Liu and Li in \cite{liu2003infinitely} used the critical point theory with Cerami condition which is weaker than the Palais-Smale condition. For more results, we refer the reader to \cite{aberqi2019existence, aberqi2017nonlinear, benslimane2020existence1, benslimane2020existence, benslimane2020existence2, benslimane2020existence3} and the references therein.\\
In Sobolev space with variable exponent Zhang and Zhao in \cite{zhang2015existence} proved the existence of strong solutions of the following $p( x )-$Laplacian Dirichlet problem via critical point theory:
$$\begin{cases}
- \, \mbox{div} ( | \nabla u |^{p( x ) - 2} \nabla u = \, f( x, u( x ) )  & \text{in \,\,$\Omega$ }, \\[0.3cm]
\, u \,  = \, 0  & \text{on \,$\partial \Omega$},
\end{cases} $$ 
and they give a new growth condition, under which, they used a new method to check the Cerami compactness condition. Hence, they proved the existence of strong solutions to the problem as above without the growth condition of the well-known Ambrossetti-Rabinowitz type and also they give some results about the multiplicity of the solutions.\\
In \cite{hurtado2018existence} and using the variational methods, the authors established the existence and multiplicity of weak solutions for a general class of quasilinear problems involving $p( . )-$Laplace type operators, with Dirichlet boundary conditions involving variable exponents without Ambrossetti and Rabinowitz type growth conditions, namely 
$$\begin{cases}
- \, \mbox{div} (a( | \nabla u |^{p( x )} ) | \nabla u |^{p( x ) - 2} \nabla u = \, \lambda\, f( x, u( x ) )  & \text{in \,\,$\Omega$ }, \\[0.3cm]
\, u \,  = \, 0  & \text{on \,$\partial \Omega$},
\end{cases} $$ 
and by different types of versions of the Mountain Pass Theorem with Cerami condition, as well as, the Fountain and Dual Theorem with Cerami condition, they obtained some existence of weak solutions for the above problem under some considerations. Moreover, they have shown that the problem treated has at least one non-trivial solution for any parameter $\lambda > 0,$ small enough, and also that the solution blows up, in the Sobolev norm, as $ \lambda \rightarrow 0^{+}.$ Finally, by suitable hypotheses on the nonlinearity $f( x, . )$ they get the existence of infinitely many weak solutions by using the Genus Theory introduced by Krasnoselskii. For a deeper comprehension, we refer the reader to \cite{ragusa2019regularity, guo2015dirichlet, zang2008p} and the references therein. As applications, we cite for examples: the study of fluid filtration in porous media, constrained heating, elastoplasticity, optimal control, financial mathematics, and others, see \cite{chen2006variable, gwiazda2008non, ruuvzivcka2004modeling, zhikov2004density} and the references therein.\\

To our knowledge, the results presented here are new, and they complement and improve the ones obtained in \cite{zang2008p, zhang2015existence, hurtado2018existence} because we are considering the general framework of Sobolev spaces with variable exponents in Complete manifolds and nonlinearities is superlinear but does not satisfy the usual Ambrossetti-Rabinowitz condition. However, we address the challenges presented by the fact that the $p(x)-$ laplacian and $ q( x )-$ laplacian operators possess more complicated nonlinearities thanthe  p-laplacian and q-laplacian operator, due to the fact that $\Delta_{p( x )}$ and $ \Delta_{q( x )}$ are not homogeneous. Moreover, we can not use Lagrange Multiplier Theorem in many problems involving this operator, which shows that our problem has more difficulty than the operators p-Laplace type.\\\\

The remainder of the paper is organized as follows. In section \ref{sec2} we will recall the definitions and some properties of Sobolev spaces with variable exponents and Sobolev spaces with variable exponents in complete manifolds. The readers can consult the following papers \cite{aubin1982nonlinear, benslimane2020existence1, benslimane2020existence, benslimane2020existence2, benslimane2020existence3, hebey2000nonlinear, trudinger1968remarks} for details. In section \ref{sec3}, using the mountain pass theorem with Cerami condition, we prove the existence of non-trivial weak solutions of problem $( \mathcal{P} ).$ Moreover, Using the Fountain theorem with Cerami condition, we demonstrate that the problem $( \mathcal{P} )$ has infinitely many (pairs) of solutions with unbounded energy.

\section{Notations and Basic Properties}\label{sec2}
In order to discuss the problem $( \mathcal{P} )$, we need some facts on spaces $W_{0}^{1, q( x )} ( \Omega ) $ where $\Omega$ is an open subset of $ \mathbb{R}^{N}$ and $ W_{0}^{1, q( x )} ( M )$ which are called the Sobolev spaces with variable exponents and the Sobolev spaces with variable exponents in complete manifolds setting. For this reason, we will recall some properties involving the above spaces, which can be found in \cite{aubin1982nonlinear, benslimane2020existence1, gaczkowski2016sobolev, hebey2000nonlinear, guo2015dirichlet} and references therein.
\subsection{Sobolev spaces with variable exponents}
Let $ \Omega$ be a bounded open subset of $ \mathbb{R}^{N} \, ( N \geq 2 )$, we define the Lebesgue space with variable exponent $ L^{q(.)} ( \Omega )$ as the set of all measurable function $ u : \Omega \longmapsto \mathbb{R} $ for which the convex modular
$$ \rho _{q(.)} ( u ) = \int_{\Omega} |\, u( x )\,|^{q( x )} \,\, dx, $$ is finite. If the exponent is bounded, i.e if $q^{+} = ess \,sup \{ \, q( x ) / x \in \Omega \, \} < + \infty,$ then the expression $$ ||\, u \,||_{q(.)} = \inf \{ \, \lambda > 0: \, \rho_{q(.)} \bigg( \frac{u}{\lambda} \bigg) \leq 1 \, \}, $$ defines a norm in $ L^{q(.)} ( \Omega )$, called the Luxemburg norm.\\ The space $ ( L^{q(.)} ( \Omega ), \, ||\,.\,||_{q(.)} )$ is a separable Banach space. Moreover, if $ 1 < q^{-} \leq q^{+} < +\infty,$ then $ L^{q(.)} ( \Omega )$ is uniformly convex, where $ q^{-} = ess\, inf \{ \, q( x ) / x \in \Omega \, \},$ hence reflexive, and its dual space is isomorphic to $ L^{q^{'}(.)} ( \Omega )$ where $ \frac{1}{q( x )} + \frac{1}{q^{'} ( x )} = 1.$\\
Finally, we have the Hölder type inequality:
$$ \bigg|\, \int_{\Omega} u\, v \,\, dx \, \bigg| \leq \bigg( \, \frac{1}{q^{-}} + \frac{1}{( q^{'} )^{-}} \bigg) \, ||\, u\,||_{q(.)} ||\, v\,||_{q^{'} (.)},$$ for all $ u \in L^{q(.)} ( \Omega )$ and $ v \in L^{q^{'}(.)} ( \Omega ).$ \\

We define the variable exponents Sobolev space by
$$ W^{1, q( x )} ( \Omega ) = \{ \, u \in L^{q( x )} ( \Omega ) \,\, \mbox{and} \,\, |\, \nabla u\,| \in L^{q( x )} ( \Omega ) \,\},$$
with the norm $$ ||\, u\,||_{W^{1, q( x )} ( \Omega )} = ||\, u\,||_{L^{q( x )} ( \Omega )} + ||\, \nabla u \||_{L^{q( x )} ( \Omega )} \,\, \forall u \in W^{1, q( x )} ( \Omega ).$$
We denote by $ W^{1, q( x )}_{0} ( \Omega ) $ the closure of $C_{0}^{\infty} ( \Omega ) $ in $ W^{1, q( x )} ( \Omega ),$ and we define the Sobolev exponent by $ q^{*} ( x ) = \frac{N \, q( x )}{N - q( x )} $ for $ q( x ) < N.$

\subsection{Sobolev spaces with variable exponents in complete manifolds}
\begin{definition}
Let $ ( M, g ) $ be a smooth Riemannain manifolds and let $ \nabla $ be the Levi-Civita connection. If $u$ is a smooth function on $M$, then $ \nabla^{k} u $ denotes the $k-$th covariant derivative of $u$, and $ | \, \nabla^{k} u \, | $ the norm of $ \nabla^{k} u $ defined in local coordinates by
$$ | \, \nabla^{k} u \, |^{2} = g^{i_{1} j_{1}} \cdots g^{i_{k} j_{k}} \, ( \nabla^{k} u )_{i_{1} \cdots i_{k}} \, ( \nabla^{k} u )_{j_{1} \cdots j_{k}} $$
where Einstein's convention is used.
\end{definition}
\begin{definition}
To define variable Sobolev spaces, given a variable exponent $q$ in $ \mathcal{P} ( M ) $ ( the set of all measurable functions $p(.) : M \rightarrow [ 1, \infty ]$ ) and a natural number $k$, introduce 
$$ C^{q(.)}_{k} ( M ) = \{ \, u \in C^{\infty} ( M ) \,\, \mbox{such that } \,\, \forall j \,\, 0 \leq j \leq k \,\, | \, \nabla^{k} u \, | \in L^{q( . ) } ( M )\, \}. $$
On $ C^{q( . )}_{k} ( M ) $ define the norm 
$$ || \, u \, ||_{L^{q( . )}_{k}} = \sum_{j = 0}^{k} || \, \nabla^{j} u \, ||_{L^{q( . )}}. $$
\end{definition}
\begin{definition}
The Sobolev spaces $ L_{k}^{q( . )} ( M ) $ is the completion of $ C^{q(.)}_{k} ( M ) $ with respect to the norm $ || \, u \, ||_{L^{q( . )}_{k}}$.
\end{definition}
\begin{definition}
Given $ ( M, g ) $ a smooth Riemannian manifold, and $ \gamma : \, [\, a, \, b \, ] \longrightarrow M $ a curve of class $ C^{1} $. The length of $ \gamma $ is 
$$ l( \gamma ) = \int_{a}^{b} \sqrt{g \, ( \, \frac{d \gamma }{d t }, \, \frac{d \gamma}{d t}\, )} \,\,dt, $$
and for a pair of points $ x, \, y \in M$, we define the distance $ d_{g} ( x, y ) $ between $x$ and $y$ by 
$$ d_{g} ( x, y ) = \inf \, \{ \, l( \gamma ) : \, \gamma: \, [ \, a, \, b \,] \rightarrow M \,\, \mbox{such that} \,\, \gamma ( a ) = x \,\, \mbox{and} \,\, \gamma ( b ) = y \, \}. $$
\end{definition}
\begin{definition}
A function $ s: \, M \longrightarrow \mathbb{R} $ is log-Hölder continuous if there exists a constant $c$ such that for every pair of points $ \{ x, \, y \} $ in $ M$ we have
$$ | \, s( x ) - s( y ) \, | \leq \frac{c}{log ( e + \frac{1}{d_{g} ( x, y )} \, )}. $$
We note by $ \mathcal{P}^{log} ( M ) $ the set of log-Hölder continuous variable exponents. The relation between $ \mathcal{P}^{log} ( M ) $ and $ \mathcal{P}^{log} ( \mathbb{R}^{N} ) $ is the following: 
\end{definition}
\begin{proposition} \cite{aubin1982nonlinear, gaczkowski2016sobolev}
Let $ q \in \mathcal{P}^{log} ( M ) $, and let $ ( \Omega, \phi ) $ be a chart such that 
$$ \frac{1}{2} \delta_{i j } \leq g_{i j} \leq 2 \, \delta_{i j } $$
as bilinear forms, where $ \delta_{i j} $ is the delta Kronecker symbol. Then $ qo\phi^{-1} \in \mathcal{P}^{log} ( \phi ( \Omega ) ).$
\end{proposition}
\begin{definition}
We say that the n-manifold $ ( M, g ) $ has property $ B_{vol} ( \lambda, v ) $ if its geometry is bounded in the following sense:\\
$ \hspace*{1cm} \bullet \,\, \mbox{The Ricci tensor of g noted by Rc ( g ) verify,} \,\,Rc ( g ) \geq \lambda ( n - 1 ) \, g $ for some $ \lambda. $\\
$ \hspace*{1cm} \bullet  $ There exists some $ v > 0 $ such that $ | \, B_{1} ( x ) \, |_{g} \geq v \,\, \forall x \in M,$ where $B_{1} ( x ) $ are the balls of radius 1 centered at some point $x$ in terms of the volume of smaller concentric balls.
\end{definition}
\begin{proposition} \cite{aubin1982nonlinear,hebey2000nonlinear} \label{prop2}
Let $ ( M, g ) $ be a complete Riemannian n-manifold. Then, if the embedding $ L^{1}_{1} ( M ) \hookrightarrow L^{\frac{n}{n - 1}} ( M )$ holds, then whenever the real numbers $q$ and $p$ satisfy $$ 1 \leq q < n, $$ and $$ q \leq p \leq q* = \frac{n q}{n - q}, $$ the embedding $ L^{q}_{1} ( M ) \hookrightarrow L^{p} ( M ) $ also holds.
\end{proposition}\label{prop5}
\begin{proposition} \cite{aubin1982nonlinear,hebey2000nonlinear}\label{prop3}
Assume that the complete n-manifold $ ( M, g ) $ has property $ B_{vol} ( \lambda, v ) $ for some $ ( \lambda, v ).$ Then there exist positive constants $ \delta_{0} = \delta_{0} ( n, \, \lambda, \, v ) $ and $ A = A ( n, \, \lambda, \, v ) $, we have, if $ R \leq \delta_{0} $, if $ x \in M $ if $ 1 \leq q \leq n $, and if $ u \in L^{q}_{1,0} ( \, B_{R} ( x ) \, ) $ the estimate 
$$ || \, u \, ||_{L^{p}} \leq A \,p \, || \, \nabla u \, ||_{L^{q}},$$ where $ \frac{1}{p} = \frac{1}{q} - \frac{1}{n}.$ 
\end{proposition}
\begin{proposition} \cite{aubin1982nonlinear,hebey2000nonlinear, gaczkowski2016sobolev} \label{prop6}
Assume that for some $ ( \lambda, v ) $ the complete n-manifold $( M, g ) $ has property $ B_{vol} ( \lambda, v ) $. Let $ p \in \mathcal{P} ( M ) $ be uniformly continuous with $ q^{+} < n.$ Then $ L^{q( . )}_{1} ( M ) \hookrightarrow L^{p( . )} ( M ) \,\, \forall q \in \mathcal{P} ( M ) $ such that $ q \ll p \ll q* = \frac{n q}{n - q}. $ In fact, for $ || \, u \, ||_{L^{q( . )}_{1}} $ sufficiently small we have the estimate $$ \rho_{p( . )} ( u ) \leq G \, ( \, \rho_{q( . )} ( u ) + \rho_{q( . )} ( | \, \nabla u \, | ) \, ), $$ where the positive constant $G$ depend on $ n, \, \lambda, \, v, \, q $ and $ p $.
\end{proposition}
\begin{proposition}\label{prop7}
Let $ u \in L^{q( x )} ( M ), \,\{ \, u_{k} \,\} \subset L^{q( x )} ( M ), \, k \in \mathbb{N},$ then we have 
\begin{enumerate}
\item[(i)]  $|| u ||_{q( x )} < 1 \,\,\mbox{( resp. = 1, $>$ 1 )} \iff \rho_{q( x )} ( u ) < 1 \,\,\mbox{( resp. = 1, $>$ 1 )},$
\item[(ii)]For $ u \in L^{q( x )} ( M ) \backslash \{ 0 \}, \,\, || u ||_{q( x )} = \lambda \Longleftrightarrow \rho_{q( x )} \big( \frac{u}{\lambda} \big) = 1.$
\item[(iii)]  $ || u ||_{q( x )} < 1 \Rightarrow || u ||_{q( x )}^{q^{+}} \leq \rho_{q( x )} ( u ) \leq || u ||_{q( x )}^{q^{-}},$
\item[(iv)]  $ || u ||_{q( x )} > 1 \Rightarrow || u ||_{q( x )}^{q^{-}} \leq \rho_{q( x )} ( u ) \leq || u ||_{q( x )}^{q^{+}},$
\item[(v)] $ \lim_{k \rightarrow + \infty} || u_{k} - u ||_{q( x )} = 0 \iff \lim_{k \rightarrow + \infty} \rho_{q( x )} ( u_{k} - u ) = 0. $
\end{enumerate}
\end{proposition}
\begin{definition} \cite{guo2015dirichlet}
The Sobolev space $ W^{1, q( x )} ( M )$ consists of such functions $ u \in L^{q( x )} ( M )$ for which $ \nabla^{k} u \in L^{q( x )} ( M )$ $k = 1, 2, \cdots, n.$ The norm is defined by 
$$ ||\, u\,||_{W^{1, q( x )} ( M )} = ||\, u\,||_{L^{q( x )} ( M )} + \sum_{k = 1}^{n} ||\, \nabla^{k} u \,||_{L^{q( x )} ( M )}.$$ 
The space $ W_{0}^{1, q( x )} ( M )$ is defined as the closure of $ C^{\infty} ( M ) $ in $ W^{1, q( x )} ( M ).$
\end{definition}
\begin{theorem}\label{theo1} \cite{guo2015dirichlet}
Let $M$ be a compact Riemannian manifold with a smooth boundary or without boundary and $ q( x ), \, p( x ) \in C( \overline{M} ) \cap L^{\infty} ( M ).$ Assume that $$ q( x ) < N , \hspace*{0.5cm} p( x ) < \frac{N\, q( x )}{N - q( x )} \,\, \mbox{for} \,\, x \in \overline{M}.$$
Then, $$ W^{1, q( x )} ( M ) \hookrightarrow L^{p( x )} ( M ),$$
is a continuous and compact embedding.
\end{theorem}
\begin{theorem}\label{theo2} 
Let $M$ be a compact Riemannian manifold with a smooth boundary or without boundary and $ q( x ), \, r( x ) \in C( \overline{M} ) \cap L^{\infty} ( M ).$ Assume that $$ q( x ) < N , \hspace*{0.5cm} r( x ) < \frac{N\, q( x )}{N - q( x )} \,\, \mbox{for} \,\, x \in \overline{M}.$$
Then, $$ W^{1, q( x )} ( M ) \hookrightarrow L^{r( x )} ( M ),$$
is a continuous and compact embedding.
\end{theorem}
\begin{proof}
The demonstration of this theorem is the same as the previous one.
\end{proof}
\begin{proposition} \cite{aubin1982nonlinear}
If $( M, g )$ is complete, then  $W^{1, q( x )} ( M ) = W^{1, q( x )}_{0} ( M ).$
\end{proposition}

Let $ D \Psi = DJ - DI: \, L^{q( x )}_{1} ( M ) \cap L^{p( x )}_{1} ( M ) \rightarrow \mbox{Hom} ( L^{q( x )}_{1} ( M ) \cap L^{p( x )}_{1} ( M ), \, \mathbb{R} )$ the differential of $ \Psi = J - I $ with the functional $J: X \rightarrow \mathbb{R} $ defined by:
$$ J( u ) = \int_{M} \frac{1}{p( x )} |\, \nabla u( x )\,|^{p( x )} \,\, dv_{g} ( x ) + \int_{M} \frac{1}{q( x )} |\, \nabla u( x )\,|^{q( x )} \,\, dv_{g} ( x )$$
and $$ I ( u ) = \int_{M} F( x, u( x ) ) \,\, dv_{g} ( x ).$$
Where $ X = W_{0}^{1, q( x )} ( M )\, \cap \, W_{0}^{1, p( x )} ( M ),$ with the norm $ || u ||_{X} = || u ||_{q( x )} + || u ||_{p( x )} \,\, \mbox{for all} \,\, x \in M.$
Then, for all $\varphi \in D( M )$ we have  
$$\langle \, J( u ), \, \varphi \,\rangle = \int_{M} |\, \nabla u( x )\,|^{p( x ) - 2} u( x ) \, \varphi ( x ) \,\, dv_{g} ( x ) + \int_{M} |\, \nabla u( x ) \,|^{q( x ) - 2} u( x )\, \varphi ( x )\,\, dv_{g} ( x ),$$
where $\langle., \, . \rangle$ denotes the usual duality between $X$ and its dual space.
\begin{lemma}
The following assumptions hold:
\begin{itemize}
\item[i/] $J$ is a continuous, bounded homeomorphism and strictly monotone operator.
\item[ii/] $J$ is a mapping of type $( S_{+} )$, that is, if $ u_{n} \rightharpoonup u$ and $$\lim \, \sup_{n \rightarrow + \infty} \langle \, J( u_{n} ) - J( u ), \, u_{n} - u \, \rangle \, \leq 0,$$ then, $u_{n} \rightarrow u. $
\end{itemize}
\end{lemma}
Next, we recall the definition of Cerami condition $( C )$ which is introduced by G. Cerami in \cite{cerami1978existence}.
\begin{definition}
Let $( E, \, ||\,.\,|| )$ be a Banach space  and $\Phi \in C^{1} ( E, \, \mathbb{R} ).$ Given $ c \in \mathbb{R},$ we say that $\Phi $ satisfies the Cerami $C$ condition ( we denote condition $( C_{c} )$), if:
\begin{itemize}
\item[$( C_{1} ) $:] Any bounded sequence ${\, u_{n} \,} \subset E$ such that  $ \Phi ( u_{n} ) \rightarrow c $ and $\Phi'( u_{n} ) \rightarrow 0$ has a convergent subsequence.
\item[$( C_{2} )$ :] There exist constant $ \delta, \, R,\, \beta > 0$ such that $$ ||\, \Psi'( u ) \,||_{E^{*}} \, ||\, u\,|| \geq \beta \,\, \mbox{for all} \,\, u \in \Psi^{-1} ( [\, c - \delta, \, c + \delta \,] ) \,\, \mbox{with} \,\, ||\, u\,|| \geq R.$$
\end{itemize}
\end{definition}
If $ \Psi \in C^{1} ( X, \, \mathbb{R} )$ satisfies condition $( C_{c} )$ for every $ c \in \mathbb{R},$ we say that $ \Psi $ satisfies condition $( C )$. \\
Let us recall the following version of mountain pass Lemma with Cerami condition with will be used in the sequel.
\begin{proposition}\label{prop1}
Let $ ( E, \, ||.|| )$ a Banach space, $\Psi \in C^{1} ( X, \, \mathbb{R} ), \,\, u_{0} \in E$ and $ \upsilon > 0,$ be such that $ || u_{0} || > \upsilon$ and $$ b = \inf_{|| u_{0} || = \upsilon} \Psi ( u ) > \Psi ( 0 ) \geq \Psi ( u_{0} ).$$ 
If $\Psi $ satisfies the condition $( C_{c} )$ with $$ c = \inf_{\gamma \in \Gamma} \max_{t \in [ 0, \, 1 ]} \Psi ( \gamma ( t ) ), \,\,\, \Gamma = \{ \gamma \in C( [ 0, \, 1 ],\, X ) \,\mbox{such as} \,\, \gamma ( 0 ) = 0, \, \gamma ( 1 ) = u_{0} \,\}.$$
Then $c$ is a critical value of $\Psi.$
\end{proposition}
\begin{remark}
Since $X$ be a reflexive and separable Banach space, there exist $\{ e_{n} \}_{n = 1}^{\infty} \subset X$ and $\{ e^{*}_{n} \}_{n = 1}^{\infty} \subset X^{*}$ such that
$$ \langle e^{*}_{n}, \, e_{m} \rangle \,= \delta_{n,\,m} = \begin{cases}
1 & \text{ if $ n = m$ }, \\[0.3cm]
0 & \text{ in $ n \neq m$ }.
\end{cases}$$
Hence, $$X = \overline{\mbox{span}} \, \{ e_{n}, \, n \geq 1 \,\} \,\, \mbox{and} \,\,X^{*} = \overline{\mbox{span}} \, \{ e^{*}_{n}, \, n \geq 1 \,\}.  $$
For $k \geq 1,$ denote $X^{k} = \mbox{span} \,\{\, e_{k} \,\}, \,\, Y_{k} = \displaystyle \bigoplus_{i = 0}^{k} X^{k}, \,\, Z_{k} = \displaystyle \overline{\bigoplus_{i = k}^{\infty} X^{k}}.$
\end{remark}\label{remark1}
Next, as in \cite{zou2001variant} we introduce the Fountain Theorem with the condition $( C )$ as follow:
\begin{theorem}\label{prop2}
Assume that $X$ is a separable Banach space, $\Psi \in C^{1} ( X, \, \mathbb{R} )$ is an even functional satisfying the Cerami condition $( C ).$ Moreover, for each $ k \geq 1$ there exist $ D_{k} > d_{k} > 0$ such as
\begin{itemize}
\item[$( A_{1} )$:] $\displaystyle \inf_{\{ u \in Z_{k} : || u || = d_{k} \}} \Psi ( u ) \rightarrow + \infty \,\, \mbox{as} \,\, k \rightarrow + \infty.$ 
\item[$( A_{2} )$:] $ \displaystyle \max_{\{ u \in Y_{k} : || u || = D_{k} \}} \Psi ( u ) \leq 0.$
\end{itemize}
Then, $\Psi$ has a sequence of critical values which tends to $+ \infty.$
\end{theorem}

\section{Existence and multiplicity results:}\label{sec3}
In this section, we state our main results and we note by $D( M )$ the space of $C^{\infty}$ functions with compact support in $M$.
\begin{definition}
$u \in X$ is said to be a non-trivial solution of the problem $( \mathcal{P} )$ if for every $ \phi \in D( M )$ we have
\begin{align*}
&\int_{M} | \nabla u( x )|^{p( x ) - 2} g( \nabla u( x ), \, \nabla \phi ( x ) ) \,\, dv_{g} ( x ) + \int_{M} | \nabla u( x ) |^{q( x ) - 2} g( \nabla u( x ), \nabla \phi ( x ) ) \,\, dv_{g} ( x )\\& - \int_{M} f( x, u( x ) ) \,.\, \phi ( x ) \,\, dv_{g} ( x ) = 0. 
\end{align*} 

Considering the energy function $ \Psi : X \rightarrow \mathbb{R} $ associated to problem $( \mathcal{P} )$ defined by 
\begin{align*}
\Psi ( u ) =& \int_{M} \frac{1}{p( x )} | \nabla u( x )|^{p( x )} \,\, dv_{g} ( x ) + \int_{M} \frac{1}{q( x )} | \nabla u( x ) |^{q( x )} \,\, dv_{g} ( x )\\& - \int_{M} F( x, u( x ) ) \,\, dv_{g} ( x ).
\end{align*}
By condition $( f_{1} )$ and Theorems \ref{theo1} and \ref{theo2}, the functional $ \Psi \in C^{1} ( X, \, \mathbb{R} )$ is well defined.\\
Moreover, for all $ \phi \in D( M )$ we have
\begin{align*}
\langle \Psi'( u ), \phi \rangle \, = &\int_{M} | \nabla u( x )|^{p( x ) - 2} g( \nabla u( x ), \, \nabla \phi ( x ) )\,\, dv_{g} ( x ) \\&+ \int_{M} | \nabla u( x )|^{q( x ) - 2} g( \nabla u( x ), \, \nabla \phi ( x ) )\,\, dv_{g} ( x ) \\&- \int_{M} f( x, u( x ) ) \,.\, \phi ( x ) \,\, dv_{g} ( x ) \hspace*{1cm} \forall u \in X 
\end{align*}
\end{definition}

\begin{lemma}\label{lemma2}
Assume that the assumptions $(f_{1} ), \, ( f_{2} ),$ and $( f_{4} )$ are satisfied. Then the functional $\Psi$ satisfies the Cerami condition $( C ).$
\end{lemma}
\begin{proof}
For any $c \in \mathbb{R},$ we first show that $ \Psi$ satisfies the assertion $( C_{1} )$ of Cerami condition $( C )$. In fact, let $\{ u_{m} \} \subset X$ be a bounded sequence such as
\begin{equation}\label{1}
\Psi ( u_{m} ) \xrightarrow{m \rightarrow + \infty} c \hspace*{0.5cm} \mbox{and} \hspace*{0.5cm} \Psi' ( u_{m} ) \xrightarrow{m \rightarrow + \infty} 0,
\end{equation}
without loss of generality, we assume that $ u_{m} \rightharpoonup u$ as $ m \rightarrow + \infty.$ By \eqref{1} we have $$ \langle \Psi' ( u_{m} ), u_{m} - u \rangle \,  \xrightarrow{m \rightarrow + \infty} 0,$$
that is 
\begin{align}\label{2}
\langle \Psi' ( u_{m} ), \,u_{m} - u \rangle \, =& \int_{M} | \nabla u_{m} |^{p( x ) -2} \nabla u_{m} \, ( \nabla u_{m} - \nabla u ) \,\, dv_{g} ( x ) \nonumber \\&+ \int_{M} | \nabla u_{m} |^{q( x ) - 2} \nabla u_{m} \, ( \nabla u_{m} - \nabla u ) \,\, dv_{g} ( x ) \nonumber \\&- \int_{M} f( x, u_{m} ( x ) )\, ( u_{m} - u )\,\, dv_{g} ( x ) \rightarrow 0,
\end{align}
as $ m \rightarrow + \infty.$\\
On the other hand, using $( f_{1} )$ and the Hölder inequality, we obtain
\begin{equation}\label{3}
\int_{M} f( x, u_{m} ( x ) ) \, ( u_{m} - u ) \,\, dv_{g} ( x ) \xrightarrow{m \rightarrow + \infty} 0.
\end{equation}
Combining \eqref{2} and \eqref{3}, we get
\begin{align*}
&\int_{M} | \nabla u_{m} ( x )|^{p( x ) - 2} \, \nabla u_{m} \, ( \nabla u_{m} - \nabla u ) \,\, dv_{g} ( x ) \\& + \int_{M} | \nabla u_{m} |^{q( x ) - 2} \, \nabla u_{m} \, ( \nabla u_{m} - \nabla u ) \,\, dv_{g} ( x ) \rightarrow 0 \,\, \mbox{as} \,\, m \rightarrow + \infty.
\end{align*}
That is 
\begin{equation}\label{4}
\langle J( u_{m} ), \, u_{m} - u \, \rangle \rightarrow 0 \,\, \mbox{as} \, \, m \rightarrow + \infty.
\end{equation}
Furthermore, since $ u_{m} \rightharpoonup u$ as $ m \rightarrow + \infty,$ from \eqref{1} we have $$ \langle \Psi'( u_{m} ), \, u_{m} - u \rangle \rightarrow 0 \,\, \mbox{as} \,\, m \rightarrow + \infty.$$
Using the same technique as before, we deduce that 
\begin{equation}\label{5}
\langle J( u ), \, u_{m} - u \rangle \rightarrow 0 \,\, \mbox{as} \,\, m \rightarrow + \infty.
\end{equation}
Hence, according to \eqref{4} and \eqref{5} we deduce that
$$ \lim_{m \rightarrow + \infty} \sup \, \langle J ( u_{m} ) - J( u ), \, u_{m} - u \, \rangle \, \leq 0. $$
Then, as $ u_{m} \rightharpoonup u$ in $X$ and since $J$ is a mapping of type $ ( S_{+} ),$ we conclude that $u_{m} \xrightarrow{m \rightarrow + \infty} u$ in $X$.\\
Now, we check that $\psi$ satisfies the assertion $( C_{2} )$ of Cerami condition $( C )$. Arguing by contradiction, there exist $ c \in \mathbb{R}$ and $\{ u_{m} \} \subset X$ satisfying:
\begin{equation} \label{6}
\Psi ( u_{m} ) \xrightarrow{n \rightarrow + \infty} c, \hspace*{0.2cm} || u_{m} ||_{X} \xrightarrow{m \rightarrow + \infty} + \infty, \hspace*{0.2cm} || \Psi'( u_{m} )||_{X^{*}} \, || u_{m} ||_{X} \xrightarrow{m \rightarrow + \infty} 0.
\end{equation}\label{7}
Let $$ \alpha_{m} = \frac{\displaystyle \int_{M} \big( | \nabla u_{m} |^{p( x )} + | \nabla u_{m} |^{q( x )} \, \big) \,\, dv_{g} ( x )}{J' ( u_{m} )},$$
choosing $ || u_{m} ||_{X} > 1,$ for $ m \in \mathbb{N},$ then we have 
\begin{align}\label{8}
c &= \lim_{m \rightarrow + \infty} \{ \Psi ( u_{m} ) - \frac{1}{\alpha_{m}} \, \langle \Psi'( u_{m} ), \, u_{m} \, \rangle \, \} \nonumber \\&= \lim_{m \rightarrow + \infty} \{ \, \frac{1}{\alpha_{m}} \, \int_{M} f( x, u_{m} ( x ) ) \,.\, u_{m} \,\, dv_{g} ( x ) - \int_{M} F( x, u_{m} ( x ) ) \,\, dv_{g} ( x ) \, \}.
\end{align}
Denote $ w_{m} = \frac{u_{m}}{|| u_{m} ||}, $ so $|| w_{m} ||_{X} = 1,$ which implies that $\{ u_{m} \}$ is bounded in $X$. \\
Hence, for a subsequence of $\{ u_{m} \}$ still denoted by $\{ w_{m} \},$ and $ w \in X,$ we obtain 
\begin{equation}\label{9}
w_{m} \rightharpoonup w \hspace*{0.3cm} \mbox{in} \,\,\, X,
\end{equation} 
\begin{equation}\label{10}
w_{m} \rightarrow w \hspace*{0.3cm} \mbox{in} \,\,\, L^{r( x )} ( M )
\end{equation}
\begin{equation}\label{11}
w_{m} ( x ) \rightarrow  w( x ) \hspace*{0.3cm} \mbox{a.e in} \,\,\, M
\end{equation}
\textbf{Step 1:} \underline{If $ w = 0$:} As in \cite{jeanjean1999existence} (Lemma 3.6) we can define a sequence $\{ t_{m} \} \subset [ 0, \, 1 ]$ such as 
\begin{equation}\label{12}
\psi ( t_{m} u_{m} ) = \max_{t \in [ 0, \, 1 ]} \psi ( t u_{m} ).
\end{equation} 
If for $ m \in \mathbb{N}, \, t_{m}$ satisfying \eqref{12} is not unique, then we choose the smaller positive value. For that, we fix $ A > \frac{1}{2 p^{+}},$ let $ \tilde{w}_{m} = ( 2 p^{+} A )^{\frac{1}{p^{-}}},$ and according to \eqref{10} we have that
$$ \tilde{w}_{m} \rightarrow 0 \,\, \mbox{in} \,\, L^{r( x )} ( M ),$$
and by $( f_{1} ),$ we have $$|\, F( x, t ) \,| \leq c \, ( 1 + |\, t\,|^{r( x )} ).$$
Hence, from the continuity of the Nemitskii operator, we get $$ F( ., \tilde{w}_{m} ) \rightarrow 0 \,\, \mbox{in} \,\, L^{1} ( M ) \,\, \mbox{as} \,\, m \rightarrow + \infty.$$
Therefore, 
\begin{equation}\label{13}
\lim_{m \rightarrow + \infty} \int_{M} F( x, \tilde{w}_{m} ) \,\, dv_{g} ( x ) = 0.
\end{equation}   
Then, for $m$ large enough, $$ \frac{( 2 p^{+} A )^{\frac{1}{p^{-}}}}{|| u_{m} ||_{X}} \in ( 0, \, 1 ), $$ and 
\begin{align*}
\Psi ( t_{m} u_{m} ) &\geq \, \Psi ( \tilde{w}_{m} ) \\& \geq \int_{M} \frac{1}{p( x )} | \nabla \tilde{w}_{m} |^{p( x )} \,\, dv_{g} ( x ) + \int_{M} \frac{1}{q( x )} | \nabla \tilde{w}_{m} |^{q( x )} \,\, dv_{g} ( x ) \\& \hspace*{0.5cm} - \int_{M} F( x, \tilde{w}_{m} ) \,\, dv_{g} ( x ) \\& \geq \frac{1}{p^{+}} \, \int_{M} ( 2 p^{+} A ) \, | \nabla w_{m} |^{p( x )} \,\, dv_{g} ( x ) + \frac{1}{q^{+}} \int_{M} ( 2 p^{+} A ) \, | \nabla w_{m} |^{q( x )} \,\, dv_{g} ( x )\\& \hspace*{0.5cm}   - \int_{M} F( x, \tilde{w}_{m} ) \,\, dv_{g} ( x ) \\& \geq 2A \int_{M} | \nabla w_{m} |^{p( x )} \,\, dv_{g} ( x ) + \frac{2 A p^{+}}{q^{+}} \int_{M} | \nabla w_{m} |^{q( x )} \,\, dv_{g} ( x )\\& \hspace*{0.5cm}  - \int_{M} F( x, \tilde{w}_{m} ) \,\, dv_{g} ( x ) \\& \geq 2 A c || w_{m} ||^{p^{+}} + \frac{2 A}{\eta \, q^{+}} || w_{m} ||^{p^{+}} - \int_{M} F( x, \tilde{w}_{m} ) \,\, dv_{g} ( x ) \\& \geq A. 
\end{align*}
That is  \begin{equation}\label{14}
\Psi ( t_{m} u_{m} ) \rightarrow + \infty.
\end{equation}
As $ \Psi ( 0 ) = 0 $ and $ \Psi ( u_{m} ) \xrightarrow{m \rightarrow + \infty} c,$ then, we have $ t_{m} \in ( 0, \, 1 )$ for $m$ large enough, and
\begin{align}\label{15}
&\int_{M} | \nabla ( t_{m} u_{m} ) |^{p( x )} \,\, dv_{g} ( x ) + \int_{M} | \nabla t_{m} u_{m} ) |^{q( x )} \,\, dv_{g} ( x ) - \int_{M} f( x, t_{m} u_{m} ) \,\, dv_{g} ( x ) \nonumber \\&= \langle \Psi'( t_{m} u_{m} ), \, t_{m} u_{m} \, \rangle = t_{m} \,\frac{\mbox{d}}{\mbox{dt}} \bigg \vert _{t = t_{m}} \Psi ( t u_{m} ) = 0.
\end{align}
Thus, from \eqref{14} and \eqref{15}, we get 
\begin{align*}
&\int_{M} \bigg( \, \frac{1}{\alpha_{t_{m}}} f( x, \, t_{m} u_{m} ) \, t_{m} u_{m} - F( x, \, t_{m} u_{m} ) \, \bigg) \,\, dv_{g} ( x )\\& = \frac{1}{\alpha_{t_{m}}} \int_{M} | \nabla t_{m} u_{m} |^{p( x )} \,\, dv_{g} ( x ) + \frac{1}{\alpha_{t_{m}}} \int_{M} | \nabla ( t_{m} u_{m} ) |^{p( x )} \,\, dv_{g} ( x ) \\& \hspace*{0.3cm}+ \frac{1}{\alpha_{t_{m}}} \int_{M} | \nabla ( t_{m} u_{m} ) |^{q( x )} \,\, dv_{g} ( x ) - \int_{M} F( x, t_{m} u_{m} ) \,\, dv_{g} ( x ) \\& = \Psi ( t_{m} u_{m} ) \xrightarrow{m \rightarrow + \infty} + \infty,
\end{align*}
where, $$ \alpha_{t_{m}} = \frac{\displaystyle\int_{M} \bigg( \, | \nabla ( t_{m} u_{m} ) |^{p( x )} + | \nabla ( t_{m} u_{m} ) |^{q( x )} \, \bigg) \,\, dv_{g} ( x )}{J' ( t_{m} u_{m} )}.$$
From the definition of $\alpha_{m}$ and $\alpha_{t_{m}}$, we have $ \alpha_{m}, \, \alpha_{t_{m}} \in [ 2 q^{-}, \, 2 p^{+} ].$ Hence, $ G_{\alpha_{m}}, \, G_{\alpha_{t_{m}}} \in \mathcal{F}.$ Then, according to $( f_{4} )$ and the fact that $$ \frac{\alpha_{t_{m}}}{\theta \, \alpha_{m}} > 0,$$ we deduce that
\begin{align*}
&\int_{M} \bigg( \, \frac{1}{\alpha_{m}} f( x, u_{m} ) \, u_{m} - F( x, u_{m} ) \, \bigg) \,\, dv_{g} ( x ) \\&= \frac{1}{\alpha_{m}} \, \int_{M} G_{\alpha_{m}} ( x, u_{m} ) \,\, dv_{g} ( x )\\& \geq \frac{1}{\theta \, \alpha_{m}} \int_{M} G_{\alpha_{t_{m}}} ( x, t_{m} u_{m} ) \,\, dv_{g} ( x ) \\& = \frac{\alpha_{t_{m}}}{\theta \alpha_{m}} \int_{M} \bigg( \frac{1}{\alpha_{t_{m}}} f( x, t_{m} u_{m} ) \, t_{m} u_{m} - F( x, t_{m} u_{m} ) \, \bigg) \,\, dv_{g} ( x ) \longrightarrow + \infty,
\end{align*} 
which contradicts \eqref{8}.\\\\
\textbf{Step 2:} \underline{If $ w \neq 0$:} From \eqref{7} we write
\begin{align}\label{16}
&\int_{M} | \nabla u_{m} |^{p( x )} \,\, dv_{g} ( x ) + \int_{M} | \nabla u_{m} |^{q( x )} \,\, dv_{g} ( x ) - \int_{M} f( x, u_{m} ) \, u_{m} \,\, dv_{g} ( x ) \nonumber \\& = \langle \Psi'( u_{m} ), \, u_{m} \rangle = o( 1 ) \, || u_{m} ||_{X},
\end{align}
then, 
\begin{align}
1 - o( 1 ) &= \int_{M} \frac{f( x, u_{m} )\,.\, u_{m}}{\displaystyle \int_{M} | \nabla u_{m} |^{p( x )} \,\, dv_{g} ( x ) + \displaystyle \int_{M} | \nabla u_{m} |^{q( x )} \,\, dv_{g} ( x )} \,\, dv_{g} ( x ) \nonumber \\& \geq \int_{M} \frac{f( x, u_{m} ) \, u_{m}}{|| u_{m} ||^{p^{+}}} \,\, dv_{g} ( x )\nonumber \\& = \int_{M} \frac{f( x, u_{m} ) \,.\, u_{m}}{| u_{m} |^{p^{+}}} \,.\, | w_{m} |^{p^{+}} \,\, dv_{g} ( x ).
\end{align}
Next, we define the set $ \mathcal{B} = \{ \, x \in M; \, w( x ) = 0 \,\}.$ So for any $ x \in \mathcal{B} \backslash \mathcal{B}_{0} = \{ \, x \in M ; \, w( x ) \neq 0 \, \},$ we have
$$ | u_{m} ( x ) | \longrightarrow + \infty \,\, \mbox{as} \,\, m \rightarrow + \infty.$$
Then, by $( f_{4} )$ we have
\begin{equation}\label{17}
\frac{f( x, u_{m} ( x ) ) \,.\, u_{m} ( x )}{| u_{m} ( x ) |^{p^{+}}} \,.\, | w_{m} ( x ) |^{p^{+}} \longrightarrow + \infty \,\, \mbox{as} \,\, m \rightarrow + \infty.
\end{equation}
Since, $| \mathcal{B} \backslash \mathcal{B}_{0} | > 0,$ we deduce via the Fatou's Lemma that
\begin{equation}\label{18}
\int_{\mathcal{B} \backslash \mathcal{B}_{0}} \frac{f( x, \, u_{m} ) \,.\, u_{m}}{| u_{m} |^{p^{+}}} \,.\, | w_{m} |^{p^{+}} \,\, dv_{g} ( x ) \longrightarrow + \infty \,\, \mbox{as} \,\, m \rightarrow + \infty.
\end{equation}
On the other hand, from $( f_{1} )$ and $( f_{4} ),$ there exists $ l > - \infty$ such as $$ \frac{f( x, t )\,.\, t}{| t |^{p^{+}}} \geq l \,\, \mbox{for} \,\, t \in \mathbb{R} \,\, \mbox{and  a.e} \,\, x \in M.$$ Moreover, it is easy to see that  $$ \int_{\mathcal{B}_{0}} | w_{m} ( x ) |^{p^{+}} \,\, dv_{g} ( x ) \longrightarrow 0.$$
Thus, there exists $ j > - \infty$ such as 
\begin{equation}\label{19}
\int_{\mathcal{B}_{0}} \frac{f( x, u_{m} ) \,.\, u_{m}}{| u_{m} |^{p^{+}}} \,.\, | w_{m} |^{p^{+}} \,\, dv_{g} ( x ) \geq l \, \int_{\mathcal{B}_{0}} | w_{m} |^{p^{+}} \,\, dv_{g} ( x ) \geq  j > - \infty.
\end{equation}
Combining \eqref{17} - \eqref{19} we get a contradiction. Thus, the functional $\psi$ satisfies the assertion $( C_{2} )$ of Cerami condition $( C ).$ This completes the proof of Lemma \ref{lemma2}.
\end{proof}
Now, we demonstrate our first existence result.
\begin{theorem}\label{theo3}
Suppose that $( f_{1} ) - ( f_{4} )$ are satisfied, and we assume that the complete n-manifold $( M, \, g )$ has property $B_{vol} ( \lambda, \, v ).$ If $ q^{+} < p^{-},$ then the problem $( \mathcal{P} )$ has at least one non-trivial solution. 
\end{theorem}
\begin{proof}
By Lemma \ref{lemma2}, $\psi$ satisfies the Cerami condition $( C )$ on $X$. Firstly, we show that the functional $\Psi$ has a geometrical structure, in order to apply Proposition \ref{prop1}. For that, we claim that there exists $\mu, \, \nu > 0$ such as $$ \Psi ( u ) \geq \mu > 0 \,\, \mbox{for any} \,\, u \in X \,\, \mbox{with} \,\, || u ||_{X} = \nu.$$ 
Let $ || u ||_{X} < 1.$ Then by Propositions \ref{prop5}, \ref{prop7} and the fact that $ q^{+} < p^{+} $we get
\begin{align}\label{20}
\Psi ( u ) &\geq \frac{c}{p^{+}} || u ||^{p^{+}}  + \frac{1}{A p q^{+}} || u ||^{q^{+}} - \int_{M} F( x, u ) \,\, dv_{g} ( x ) \nonumber \\& \geq c^{*} || u ||^{q^{+}}_{X} - \int_{M} F( x, u( x ) ) \,\, dv_{g},
\end{align}
 with $ c^{*} = \min \, \{ \frac{c}{p^{+}}, \, \frac{1}{Apq^{+}} \, \}.$
 According to theorems \ref{theo1} and \ref{theo2}, there exist two positive constants $c_{1}, \, c_{2} > 0$ such as $$ | u |_{p^{+}} \leq c_{1} || u ||_{X} \,\, \mbox{and} \,\, | u |_{r( x )} \leq || u ||_{X} \,\, \mbox{for all} \,\, u \in X.$$
 Let $ \epsilon > 0$ be small enough, such as $$ \epsilon \, c_{1}^{p^{+}} < \frac{c^{*}}{2}.$$
 According to $( f_{1} )$ and $( f_{2} )$, we have $$ F( x, t ) \leq \epsilon \, | t |^{p^{+}} + c_{\epsilon} \, | t |^{r( x )} \,\, \mbox{for all} \,\, ( x, t ) \in M \times \mathbb{R},$$
 for $|| u || \leq 1,$ we get 
 \begin{align*}
 \Psi ( u ) & \geq c^{*} \, || u ||^{q^{+}}_{X} - \epsilon \, \int_{M} | u |^{p^{+}} \,\, dv_{g} ( x ) - c_{\epsilon} \, \int_{M} | u( x ) |^{r( x )} \,\, dv_{g} ( x ) \\& \geq c^{*} \, || u ||_{X}^{q^{+}} - \epsilon \, || u ||_{p^{+}}^{p^{+}} - c_{\epsilon} \, || u ||_{r( x )}^{r^{-}} \\& \geq c^{*} \, || u ||_{X}^{q^{+}} - \epsilon \, c_{1}^{p^{+}} \, || u ||_{X}^{p^{+}} - c_{\epsilon} \, c_{2}^{r^{-}} \, || u ||_{X}^{r^{-}}.
 \end{align*}
Since, $ q^{+} < p^{+} < r^{-},$ then there exist two positive real number $\mu$ and $\nu$ such as $$ \Psi ( u ) \geq \mu > 0 \,\, \mbox{for all} \,\, u \in X \,\, \mbox{with} \,\, || u ||_{X} = \nu.$$
On the other hand, we affirm that there exists $u_{0} \in X \backslash \overline{\mathcal{B}_{0} ( \nu )}$ such as 
\begin{equation}
\Psi ( u ) < 0.
\end{equation}
Let $ \phi_{0} \in X \backslash \{ 0 \},$ by $( f_{4} )$ we can choose a constant
$$ \delta > \frac{  \displaystyle \frac{1}{p^{-}} \,\int_{M} | \nabla \phi_{0} |^{p( x )} \,\, dv_{g} ( x ) - \frac{1}{q^{-}} \int_{M} | \nabla \phi_{0} |^{q( x )} \,\, dv_{g} ( x )}{ \displaystyle\int_{M} | \phi_{0} |^{p^{+}} \,\, dv_{g} ( x )},$$
and a constant $c_{\delta} > 0$ depending on $ \delta$ such as
$$ F( x, t ) \geq \delta \, | t |^{p^{+}} \,\, \mbox{for all} \,\, | t | > c_{\delta} \,\, \mbox{and uniformly in} \,\, M.$$
Let $ k > 1$ be large enough, we have
\begin{align*}
\Psi ( k \, \phi_{0} ) &= \int_{M} \frac{1}{p( x )} | \nabla ( k \, \phi_{0} ) |^{p( x )} \,\, dv_{g} ( x ) + \int_{M} \frac{1}{q( x )} | \nabla ( k \, \phi_{0} ) |^{q( x )} \,\, dv_{g} ( x ) \\& \hspace*{0.3cm}- \int_{M} F( x, k\, \phi_{0} ) \,\, dv_{g} ( x ) \\& \geq \frac{k^{P^{+}}}{p^{-}} \int_{M} | \nabla \phi_{0} |^{p( x )} \,\, dv_{g} ( x ) + \frac{k^{q^{+}}}{q^{-}} \int_{M} | \nabla \phi_{0} |^{q( x )} \,\, dv_{g} ( x )\\& \hspace*{0.3cm} - \int_{\{\, | k \phi_{0} | > c_{\delta} \,\}} F( x, k \phi_{0} ) \,\, dv_{g} ( x ) - \int_{\{\, | k \phi_{0} | < c_{\delta} \,\}} F( x, k \phi_{0} ) \,\, dv_{g} ( x ) \\& \geq \frac{k^{P^{+}}}{p^{-}} \int_{M} | \nabla \phi_{0} |^{p( x )} \,\, dv_{g} ( x ) + \frac{k^{q^{+}}}{q^{-}} \int_{M} | \nabla \phi_{0} |^{q( x )} \,\, dv_{g} ( x )\\& \hspace*{0.3cm} - \int_{\{\, | k \phi_{0} | \leq c_{\delta} \,\}} F( x, k \phi_{0} ) \,\, dv_{g} ( x ) - \delta \, k^{p^{+}} \int_{M} | \phi_{0} |^{p^{+}} \,\, dv_{g} ( x ) \\& \hspace*{0.3cm} + \delta \, \int_{\{ | k \phi_{0} | \leq c_{\delta}\}} | k \phi_{0} |^{p^{+}} \,\, dv_{g} ( x ) \\& \geq \frac{k^{P^{+}}}{p^{-}} \int_{M} | \nabla \phi_{0} |^{p( x )} \,\, dv_{g} ( x ) + \frac{k^{q^{+}}}{q^{-}} \int_{M} | \nabla \phi_{0} |^{q( x )} \,\, dv_{g} ( x )\\& \hspace*{0.3cm} - \delta \, k^{p^{+}} \int_{M} | \phi_{0} |^{p^{+}} \,\, dv_{g} ( x ) + c_{5},
\end{align*}
which implies that $$ \Psi ( k \phi_{0} ) \longrightarrow - \infty \,\, \mbox{as} \,\, k \rightarrow + \infty.$$
Then, there exists $k_{0} > 0$ and $u_{0} = k_{0} \, \phi_{0} \in X_{0} \backslash \overline{\mathcal{B}_{\nu} ( 0 )}$ such as \eqref{20} hold true.\\
Thereby, proposition \ref{prop1} guarantees that problem $( \mathcal{P} )$ has at least a non-trivial weak solution. This completes the proof.
\end{proof}
\begin{theorem}\label{theo6}
Assume that $( f_{1} ), \, ( f_{3} ), \, ( f_{4} ), \,( g )$ hold and we assume that the complete n-manifold $( M, \, g )$ has property $B_{vol} ( \lambda, \, v ).$ If $q^{-} > p^{+},$ then the problem $( \mathcal{P} )$ has a sequence of weak solutions with unbounded energy.
\end{theorem}
\begin{proof}
we will divide the proof of this theorem into two steps. In the first step, we will demonstrate that the problem $( \mathcal{P} )$ acquires a sequence of weak solutions $\{ \pm u_{m} \}_{m = 1}^{\infty}$ such as $$ \Psi ( \pm u_{m} ) \longrightarrow + \infty \,\, \mbox{as} \,\, m \rightarrow + \infty.$$ In the second step, we will prove that if $k$ is large enough, then there exist $ D_{k} > d_{k} > 0$ such as the assertions $( A_{1} )$ and $( A_{2} )$ are satisfied.\\
\textbf{Step 1:} The proof is based on the Fountain Theorem (given by Theorem \ref{prop2}). Indeed, from $( f_{3} ),\, \Psi $ is an even functional. And from Lemma \ref{lemma2}, $\Psi $ satisfies the condition $( C ).$\\
For that, we will use the mean value theorem in the following form: For every $\beta \in C_{+} ( \overline{M} ) = \{ \beta \in C( \overline{M} ), \, \beta ( x )  > 1 \,\, \forall x \in M \}$ and $ u \in L^{\beta ( x )} ( M ),$ there exist $\zeta \in M$ such that
\begin{equation}\label{22}
\int_{M} | u( x ) |^{\beta ( x )} \,\, dv_{g} ( x ) = | u |^{\beta ( \zeta )}_{\beta ( x )}.
\end{equation}
Indeed, on the one hand, by proposition \ref{prop7}, it is easy to see that $$ \rho_{\beta ( x )} \bigg( \frac{u}{|| u ||_{\beta ( x )}} \bigg) = \int_{M} \bigg( \frac{| u |}{|| u ||_{\beta ( x )}} \bigg)^{\beta ( x )} \,\, dv_{g} ( x ) = 1.$$
On the other hand, by the mean value theorem for integrals, there exists a positive constant $\bar{\beta} \in [ \beta^{-}, \, \beta^{+} ]$ depending on $\beta$ such as
$$ \int_{M} \bigg( \frac{| u |}{|| u ||_{\beta ( x )}} \bigg)^{\beta ( x )} \,\, dv_{g} ( x ) = \bigg( \frac{1}{|| u ||_{\beta ( x )}} \bigg)^{\bar{\beta}} \, \int_{M} | u |^{\beta ( x )} \,\, dv_{g} ( x ).$$
Moreover, the continuity of $\beta$ ensures that there exists $ \zeta \in M $ such as $ \beta ( \zeta ) = \bar{\beta}.$ Combining this fact with the above inequalities, we get \eqref{22}.\\
\textbf{Step 2:} $( A_{1} ):$ For all $ u \in Z_{k} $ such as $|| u ||_{X} = d_{k}$ ( $d_{k}$ will be specified below ), by $( f_{1} ),$\eqref{22} and Proposition \ref{prop5} we obtain
\begin{align*}
\Psi ( u ) &= \int_{M} \frac{1}{p( x )} \, | \nabla u( x ) |^{p( x )} \,\, dv_{g} ( x ) + \int_{M} \frac{1}{q( x )} | \nabla u( x ) |^{q( x )} \,\, dv_{g} ( x )  \\& \hspace*{0.3cm}- \int_{M} F( x, u( x ) ) \,\, dv_{g} ( x ) \\& \geq \bigg( \frac{1}{p^{+}} + \frac{1}{A p q^{+}} \bigg) \, || u ||_{X}^{p^{-}} - c_{5}\, || u ||_{r( x )}^{r( \zeta )} - c_{6} \, || u ||_{X} \,\,\, \mbox{where} \,\,
\, \zeta \in M\\& \geq \begin{cases} \big( \frac{1}{p^{+}} + \frac{1}{A p q^{+}} \big) \,|| u ||_{X}^{p^{-}} - c_{5} - c_{6} \, || u ||_{X} & \text{if \, $|| u ||_{r( x )} \leq 1$} \\[0.3cm]
\big( \frac{1}{p^{+}} + \frac{1}{A p q^{+}} \big) \,|| u ||_{X}^{p^{-}} - c_{5} ( \eta_{k} \, || u ||_{X} )^{r^{+}} - c_{6} \, || u ||_{X} & \text{if \, $|| u ||_{r( x )} > 1$} \end{cases} \\& \geq \bigg( \frac{1}{p^{+}} + \frac{1}{A p q^{+}} \bigg) \, || u ||_{X}^{p^{-}} - c_{5} ( \eta_{k} \, || u ||_{X} )^{r^{+}} - c_{6} \, || u ||_{X} - c_{5} \\& \geq d_{k}^{p^{-}} \, \bigg( \frac{1}{p^{+}} + \frac{1}{A p q^{+}} - c_{5} \eta_{k}^{r^{+}} d_{k}^{r^{+} - p^{-}} \, \bigg) - c_{6} \, d_{k} - c_{5}.
\end{align*}
We fix $d_{k}$ as follows $$ d_{k} = ( r^{+} \, c_{5} \, \eta_{k}^{r^{+}} )^{\frac{1}{p^{-} - r^{+}}}.$$
Then, $$ \Psi ( u ) \geq d_{k}^{p^{-}} \, \bigg( \frac{1}{p^{+}} + \frac{1}{A p q^{+}} - \frac{1}{r^{+}} \,  \bigg) - c_{6} \, d_{k} - c_{5}.$$
According to Lemma 3.4 in \cite{zhang2015existence}. We know that $ \displaystyle \lim_{k \rightarrow + \infty} \theta ( x ) = 0.$ Then, since $ 1 < q^{+} < p^{-} \leq p^{+} < r^{+},$ we conclude that $$ d_{k} \longrightarrow + \infty \,\, \, \mbox{as} \,\,\, k \rightarrow + \infty.$$
Thus, $$ \Psi ( u ) \longrightarrow + \infty \,\,\, \mbox{as} \,\,\, || u || \rightarrow + \infty \,\,\, \mbox{with} \,\,\, u \in Z_{k}.$$
Which means that the assertion $( A_{1} )$ is verified.\\
$( A_{2} ):$ According to Remark \ref{remark1}, we have $\mbox{dim} Y_{k} < + \infty,$ and since all norms are equivalent in finite dimensional space, then there exists $ h_{k} > 0,$ for all $ u \in Y_{k}$ with $ || u ||_{X} $ is big enough, we obtain
\begin{align*}
J( u ) &\leq \frac{1}{p^{-}} \int_{M} | \nabla u( x ) |^{p( x )} \,\, dv_{g} ( x ) + \frac{1}{q^{-}} \int_{M} | \nabla u( x ) |^{q( x )} \,\, dv_{g} ( x ) \\& \leq \frac{1}{p^{-}} || u ||_{p( x )}^{p^{+}} + \frac{1}{q^{-}} || u ||_{q( x )}^{q^{+}} \\& \leq \frac{c_{7}}{p^{-}} || u ||_{X}^{p^{+}} + \frac{c_{8}}{q^{-}} || u ||_{X}^{q^{+}},
\end{align*} 
since $ q^{+} < p^{+},$ we get 
\begin{equation}\label{23}
 J( u ) \leq \bigg( \frac{c_{7}}{p^{-}} + \frac{c_{8}}{q^{-}} \bigg) || u ||_{X}^{p^{+}} = h_{k} \, || u ||_{p^{+}}^{p^{+}}.
\end{equation}
Next, according to $( f_{2} )$, there exists  $B_{k} > 0$ such as for any $| t | \geq B_{k},$ we have $$ F( x, t ) \geq 2 \, h_{k} \, | t |^{p^{+}} \,\,\, \mbox{for any} \,\,\, x \in M.$$ Moreover, by $( f_{1} ),$ there exist a positive $L_{k}$ such as $$ F( x, t ) \leq L_{k} \,\,\, \mbox{for all} \,\,\, ( x, t ) \in M \times [ - B_{k}, \, B_{k} ].$$
Then, for every $( x, t ) \in M \times \mathbb{R}$ we conclude that
\begin{equation}\label{24}
F( x, t ) \geq 2 \, h_{k} \, | t |^{p^{+}} - L_{k}.
\end{equation}
Combining \eqref{23} and \eqref{24}, for all $ y \in Y_{k} $ such as $ || u ||_{X} = D_{k} > d_{k}$ we infer that
\begin{align*}
\Psi ( u ) &= J( u ) - I( u ) \\& \leq h_{k} \, || u ||_{p^{+}}^{p^{+}} - 2 \, h_{k} \, || u ||_{p^{+}}^{p^{+}} + L_{k} \\& \leq - h_{k} \, c_{9} \, || u ||_{X}^{p^{+}} + L_{k} \, | M |,
\end{align*}
where $| M |$ is the measure of $M$ and $ | M | < \infty.$\\
Hence, for $D_{k}$ large enough $ ( D_{k} > d_{k} ),$ we obtain that $$ \max_{\{ u \in Y_{k}: \, || u ||_{X} = D_{k} \}} \Psi ( u ) = 0.$$
which implies that the assertion $( A_{2} )$ holds. Then by applying the Fountain Theorem, we achieve the proof of Theorem \ref{theo6}.
\end{proof}
We will end this section with a suitable example, that checks all the above conditions and
Theorems,
\begin{example}
For $ f( x, t ) = \displaystyle \frac{t\,| t |^{\alpha( x ) - 2}}{\mbox{Log} \,( 1 + | t | )},$ if $ t \neq 0$ and $ f( x, 0 ) = 0$ with $ 2 < p^{+} \leq \alpha( x ) < r( x )$ for all $ x \in \overline{M}.$  So it is easy to see that the function$f$ does not satisfy the Ambrosetti-Rabinowitz condition, but it satisfies our assumptions $( f_{1} ) - ( f_{5} )$. Then, our problem $( \mathcal{P} )$ becomes
$$ ( \mathcal{P} ) \, \begin{cases}
- \, \Delta_{p( x )} u( x ) - \Delta_{q( x )} u( x ) = \, \frac{t\,| t |^{\alpha( x ) - 2}}{\mbox{Log} \,( 1 + | t | )}  & \text{in \,\,M }, \\[0.3cm]
\, u \,  = \, 0  & \text{on \,$\partial$M },
\end{cases} $$ 
Consequently, the results corresponding to Theorems \ref{theo3} and \ref{theo6} can be achieved and still be true for the problem $( \mathcal{P} ).$
\end{example}

%For acknowledgements section, please don't number the section, please begin it with \section*{Acknowledgements}
\section*{Acknowledgments} The authors would like to thank the anonymous referees for the valuable suggestions and comments which improved the quality of the presentation.

% You may incorporate your references as follows in your main tex file.
% Using BibTex is not recommended but can be handled.

\end{document}